\documentclass[12pt,a4paper]{article}
\usepackage{amsfonts,amsmath,amsthm,amssymb,mathtools,blkarray,bbm,fullpage,enumitem}
\usepackage[hidelinks]{hyperref}
\usepackage[capitalize,nameinlink]{cleveref}
\usepackage[numbers,sort]{natbib}

\usepackage{todonotes}

\newtheorem{theorem}{Theorem}[section]
\newtheorem{lemma}[theorem]{Lemma}

\newtheorem{conjecture}[theorem]{Conjecture}
\newtheorem{question}[theorem]{Question}
\theoremstyle{definition}

\newtheorem*{remark}{Remark}

\usepackage{xurl} 

\DeclareMathOperator{\tr}{tr}
\DeclareMathOperator{\pr}{Pr}

\newcommand{\ab}[1]{\lvert #1 \rvert}

\newcommand{\E}{\mathbb{E}}

\crefname{lemma}{Lemma}{Lemmas}
\newlist{thmenum}{enumerate}{1}
\setlist[thmenum]{label=(\roman*), ref=\thetheorem(\roman*)}
\crefalias{thmenumi}{theorem} 
\newlist{lemenum}{enumerate}{1}
\setlist[lemenum]{label=(\roman*), ref=\thetheorem(\roman*)}
\crefalias{lemenumi}{lemma} 

\title{The inertia bound is far from tight}
\author{Matthew Kwan\thanks{Institute of Science and Technology Austria (ISTA), 3400 Klosterneuburg, Austria. Email address:  \texttt{matthew.kwan@ist.ac.at}. Supported by ERC Starting Grant ``RANDSTRUCT'' No.\ 101076777.} \and Yuval Wigderson\thanks{Institute for Theoretical Studies, ETH Z\"urich, 8092 Z\"urich, Switzerland. Email address: \texttt{yuval.wigderson@eth-its.ethz.ch}. Supported by Dr.\ Max R\"ossler, the Walter Haefner Foundation, and the ETH Z\"urich Foundation.}}
\date{}

\begin{document}

\maketitle

\begin{abstract}
The {inertia bound} and {ratio bound} (also known as the {Cvetkovi\'c bound} and {Hoffman bound}) are two fundamental inequalities in spectral graph theory, giving upper bounds on the independence number $\alpha(G)$ of a graph $G$ in terms of spectral information about a weighted adjacency matrix of $G$. For both inequalities, given a graph $G$, one needs to make a judicious choice of weighted adjacency matrix to obtain as strong a bound as possible.

While there is a well-established theory surrounding the ratio bound, the inertia bound is much more mysterious, and its limits are rather unclear. In fact, only recently did Sinkovic find the first example of a graph for which the inertia bound is not tight (for any weighted adjacency matrix), answering a longstanding question of Godsil. We show that the inertia bound can be extremely far from tight, and in fact can significantly underperform the ratio bound: for example, one of our results is that for infinitely many $n$, there is an $n$-vertex graph for which even the unweighted ratio bound can prove $\alpha(G)\le 4n^{3/4}$, but the inertia bound is always at least $n/4$. In particular, these results address questions of Rooney, Sinkovic, and Wocjan--Elphick--Abiad.
\end{abstract}

\section{Introduction}
Spectral graph theory contains a wide array of deep and surprising results which relate certain combinatorial graph parameters to linear-algebraic parameters of associated matrices. Of particular importance are those results which bound the independence number $\alpha(G)$ of a graph $G$ in terms of its spectrum, as such results have many applications in other areas of combinatorics
(see e.g.\ the monographs \cite{GoMe,GR01,BH12,BCN89}).

Probably the most famous such result is the \emph{ratio bound} (also known as the \emph{Hoffman bound}). To state it, we need some notation. If $G$ is a graph with vertex set $\{1,\dots,n\}$, we say that $A$ is a \emph{weighted adjacency matrix} of $G$ if $A_{ij}=0$ whenever $ij$ is not an edge of $G$. In other words, we can obtain $A$ by starting with the adjacency matrix of $G$, and replacing every $1$ by an arbitrary real number (including zero and negative numbers), while maintaining symmetry of the matrix.
\begin{theorem}[Ratio bound]\label{thm:ratio}
    Let $G$ be a graph and let $A$ be a weighted adjacency matrix of $G$ with equal row sums. Let $\lambda_{\mathrm{min}}$ and $\lambda_{\mathrm{max}}$ be the minimum and maximum eigenvalues of $A$, respectively. Then
    \[
    \alpha(G) \leq \left|\frac{\lambda_{\mathrm{min}}}{\lambda_{\mathrm{max}}-\lambda_{\mathrm{min}}}\right|n.
    \]
\end{theorem}
The ratio bound was first proved by Hoffman (unpublished), who only stated it in the case that $A$ is the (ordinary) adjacency matrix of $G$. Even this result is surprisingly powerful; for example, it can be used to give a short proof of the Erd\H os--Ko--Rado theorem~\cite{EKR} (see e.g.\ \cite[Section~2.5]{GoMe} for details). However, the flexibility of choosing an arbitrary weighted adjacency matrix makes the ratio bound much more powerful (see e.g.\ \cite{Wil84}).
Determining the optimal weights to use for the ratio bound can be formulated as a semidefinite program, and the rich theory of semidefinite optimization can therefore be used to theoretically and computationally determine the optimal choice of weights for any given graph.

In this paper we will mostly be concerned with a closely related, yet much more mysterious, bound, known as the \emph{inertia bound} (or the \emph{Cvetkovi\'c bound}). For a symmetric $n \times n$ matrix $A$, we denote by $n_{\ge 0}(A)$ the number of non-negative eigenvalues of $A$. 
\begin{theorem}[Inertia bound]\label{thm:inertia}
    Let $G$ be a graph and let $A$ be a weighted adjacency matrix of $G$. Then\footnote{In the literature, the inertia bound is often stated as $\alpha(G) \leq \min \{n-n_{>0}(A), n-n_{<0}(A)\}$, but it is easy to see that this statement is equivalent to \cref{thm:inertia}, by replacing $A$ with $-A$ if it has more negative than positive eigenvalues.}
    \[
    \alpha(G) \leq n_{\geq0}(A).
    \]
\end{theorem}
In many ways, the story of the inertia bound parallels that of the ratio bound. It was first proved by Cvetkovi\'c~\cite{Cve71}, who stated it only for adjacency matrices. This already has a number of interesting applications, including another short proof of the Erd\H os--Ko--Rado theorem (see e.g.\ \cite[Section~2.10]{GoMe} for details), but the flexibility of general weighted adjacency matrices makes the bound much more powerful. The general statement for arbitrary weighted adjacency matrices was perhaps first noted by Calderbank and Frankl \cite{CaFr}, who used it to prove several results in extremal set theory.

Compared to the ratio bound, we know very little about optimal weight matrices for the inertia bound: the search space is infinite, and it is unclear how to minimize $n_{\geq 0}$ over this search space (though see \cite{Roo14} for some heuristics). In many specific applications, such as recent work of Huang--Klurman--Pohoata \cite{HKP20} and the Calderbank--Frankl result mentioned above \cite{CaFr}, the choice of weights can be guided by the symmetries inherent in the problem. Specifically, in such applications, one can restrict to a natural subset of the space of all weighted adjacency matrices, called the \emph{Bose--Mesner algebra}, 
and apply representation-theoretic techniques to understand the structure of this algebra (see e.g.\ \cite{GoMe,BCN89} for introductions to this theory).
Similarly, Huang's breakthrough resolution of the sensitivity conjecture \cite{Hua19} used ideas closely related to the inertia bound, and the choice of weights comes from the inherent symmetries of the hypercube (see \cite{Bis19,AlZh,HLL,Mat22,Tao19,Ihr19,Kar19} for discussion on the connections to the inertia bound, and on the theory behind the choice of weights). For more on the general relationships between graphs and spaces of matrices supported on their edges, see \cite{LQWWZ} and references therein.

Due to our lack of understanding of optimal weighted adjacency matrices, it is unclear what the limits of the inertia bound really are, even for specific small graphs. The most fundamental question in this direction is the following (which seems to have been first explicitly asked by Godsil~\cite{God04}, and reiterated in \cite{EG10,Elz07,Roo14}).
\begin{question}[Godsil \cite{God04}]\label{qu:godsil}
Is the inertia bound always tight? In other words, is it the case that for every graph $G$, there exists a weighted adjacency matrix $A$ with $\alpha(G)=n_{\geq 0}(A)$?
\end{question}

Godsil's question was open for more than a decade, until it was finally resolved in the negative by Sinkovic \cite{Sin18}.
\begin{theorem}[Sinkovic \cite{Sin18}]\label{thm:sinkovic}
    Let $G$ be the \emph{Paley graph} on $17$ vertices. Then $\alpha(G)=3$, but $n_{\geq 0}(A) \geq 4$ for every weighted adjacency matrix $A$ of $G$.
\end{theorem}
Sinkovic's proof involves a great deal of casework, and as such is highly specific to the 17-vertex Paley graph. However, in recent years, new techniques were developed which shed more light on Godsil's question. First, Man\v cinska and Roberson \cite{MaRo} introduced a new graph parameter, the \emph{quantum independence number} $\alpha_q$, which satisfies $\alpha_q(G) \geq \alpha(G)$ for every graph $G$. Moreover, they proved \cite[Corollary 4.8]{MaRo} that for infinitely many $n$, there exists an $n$-vertex graph with independence number $O(n^{1-\varepsilon})$ and quantum independence number $\Omega(n/{\log n})$, for some small absolute constant $\varepsilon>0$. More recently, Wocjan, Elphick and Abiad~\cite{WEA22} proved that the inertia bound is also an upper bound on the quantum independence number, even when one allows \emph{Hermitian} weights. Namely, they proved \cite[Theorem 3.3]{WEA22} that if $A$ is a Hermitian weighted adjacency matrix\footnote{Here, we say that $A$ is a Hermitian weighted adjacency matrix of $G$ if it is a Hermitian matrix with complex entries such that $A_{ij}=0$ whenever $ij \notin E(G)$.}  of $G$, then $\alpha_q(G) \leq n_{\geq 0}(A)$. Putting these results together, one obtains the following, stronger, negative answer to \cref{qu:godsil}.
\begin{theorem}[{\cite{MaRo,WEA22}}]\label{thm:quantum}
    There exist constants $C,c,\varepsilon>0$ such that the following holds. For infinitely many values of $n$, there exists an $n$-vertex graph with 
    \[
    \alpha(G) \leq Cn^{1-\varepsilon} \qquad \text{ and } \qquad n_{\geq 0} (A) \geq \alpha_q(G) \geq c \frac{n}{\log n}
    \]
    for every Hermitian weighted adjacency matrix $A$ of $G$.
\end{theorem}

\cref{thm:quantum,thm:sinkovic} provide answers to \cref{qu:godsil}, but this is really just the tip of the iceberg. The most obvious follow-up question (raised by Rooney~\cite{Roo14} and Sinkovic~\cite{Sin18}) is how large the gap in \cref{thm:inertia} can be; \cref{thm:quantum} implies that this gap can be as large as $\Omega(n/{\log n})$, but perhaps it can be even larger.

Second, there is the quantum analogue of \cref{qu:godsil}, that is, the question of whether the inertia bound\footnote{Here, and from now on, when we say ``the inertia bound'' we will usually mean the minimum of $n_{\ge 0}(A)$ over all weighted adjacency matrices $A$ of a given graph.} is always tight for the quantum independence number. This question was answered by Wocjan, Elphick, and Abiad \cite{WEA22}, who gave an example of an 18-vertex graph $G$ with $\alpha_q(G) < n_{\geq 0}(A)$ for all Hermitian weighted adjacency matrices $A$ of $G$. However, this example does not rule out the possibility that every graph has a weighted adjacency matrix for which the gap between the quantum independence number and the inertia bound is always at most 1.

Third, even though we now know that the inertia bound is not always tight, perhaps it is still the case that the inertia bound is in some sense ``the best possible spectral bound'' on the independence number. Concretely, is the inertia bound always stronger than the ratio bound? Actually, this question is closely related to the previous one, since the ratio bound is also an upper bound on the quantum independence number of a graph~\cite[Corollary 4.9]{MaRo}.

As our main result, we address all the above questions, showing that in fact the inertia bound can dramatically underperform the ratio bound, and thereby also obtaining new bounds on the largest possible gap between the inertia bound and independence number.
\begin{theorem}\label{thm:intro main}
For infinitely many positive integers $n$, there exists an $n$-vertex regular graph $G$, whose (ordinary) adjacency matrix has minimum and maximum eigenvalues $\lambda_{\mathrm{min}}$ and $\lambda_{\mathrm{max}}$, respectively, such that 
    \[
    \alpha(G)\le \alpha_q(G)\le \left|\frac{\lambda_{\mathrm{min}}}{\lambda_{\mathrm{max}}-\lambda_{\mathrm{min}}}\right|n \leq 4n^{\frac 34}\qquad \text{ and } \qquad n_{\geq 0}(A) \geq \frac n4
    \]
    for every Hermitian weighted adjacency matrix $A$ of $G$.
\end{theorem}

In fact, our main technical result is substantially stronger than \cref{thm:intro main}: it shows that the inertia bound is very weak whenever $G$ is $C_4$-free, i.e.\ does not contain $C_4$ as a subgraph. Moreover, we can obtain a small constant-factor improvement when $G$ is also $C_3$-free, i.e.\ has girth at least 5.

\begin{theorem}\label{thm:main}
Let $G$ be an $n$-vertex graph and let $A$ be any Hermitian weighted adjacency matrix of $G$.
\begin{thmenum}
\item \label{thmit:C4-free}If $G$ is $C_4$-free, then
    \[
    n_{\ge 0}(A) \geq \beta n,
    \]
    where $\beta = \sqrt 3 - \frac 32 \approx 0.232$.
\item \label{thmit:girth-5}If $G$ has girth at least 5, then
    \[
    n_{\ge 0}(A) \geq \frac n4.
    \]
\end{thmenum}
\end{theorem}
\cref{thm:main} immediately implies the existence of graphs for which the inertia bound is far from tight, by letting $G$ be a $C_4$-free graph whose independence number is $o(n)$. There are many examples of such graphs, defined explicitly or obtained via the probabilistic method. In particular, the statement of \cref{thm:intro main} follows from \cref{thmit:girth-5}, taking $G$ to be an edge-transitive induced subgraph of at least half the vertices of the polarity graph of a finite projective plane, with girth at least 5 (such a graph was shown to exist by Parsons~\cite[Theorem 1(a)]{Par76})\footnote{By \cite[Theorem 9]{Lov79}, the Lov\'asz theta number of $G$ is then $|\lambda_{\mathrm{min}}/(\lambda_{\mathrm{max}}-\lambda_{\mathrm{min}})|n$, where $\lambda_{\mathrm{min}},\lambda_{\mathrm{max}}$ are the minimum and maximum eigenvalues, respectively, of the ordinary adjacency matrix. The Lov\'asz theta number is monotonic with respect to induced subgraphs, and it is well-known (e.g.\ \cite{MuWi}) that the theta number of the full polarity graph $G_0$ is at most $2\ab{V(G_0)}^{3/4}$.}. It is an interesting byproduct of our results that the inertia bound is so weak for such algebraically-defined graphs, as generally speaking, techniques from spectral graph theory seem particularly well-suited to highly structured graphs with many symmetries.

We also remark that if one is interested in the largest possible {multiplicative} gap between the independence number and inertia bound, one can combine \cref{thmit:C4-free} with a result of Bohman and Keevash \cite{BoKe}: they used the \emph{random $C_4$-free process} to prove the existence of $C_4$-free graphs with independence number $O((n\log n)^{\frac 23})$. It is a major open problem, asked repeatedly by Erd\H os (e.g.\ \cite{MR975526,MR777160,MR804676,MR602413}), to determine whether there exist $C_4$-free graphs with independence number $n^{\frac 12 + o(1)}$.
We remark too that the constant factors $\beta$ and $\frac14$ in \cref{thm:main} appear to be the best possible with our technique, although there is no reason to believe they are optimal.

Although our proof is quite short, we end this introduction with a high-level sketch of the argument. Let $A$ be a weighted adjacency matrix of a $C_4$-free graph $G$. If we let the eigenvalues of $A$ be $\lambda_1 \geq \dotsb \geq \lambda_n$, and let $X$ be a random variable taking on value $\lambda_i$ with probability $1/n$ (this is the \emph{empirical spectral distribution}), then our task reduces to proving a lower bound on $\pr(X \geq 0)$. To this end, it suffices to study the \emph{moments} of this random variable: it turns out that $\pr(X \geq 0)$ is reasonably large whenever $\E [X^4]$ is not too much larger than $(\E [X^2])^2$ (\cref{lem:anticoncentration}). So, it suffices to prove an upper bound on the fourth moment of $X$. Using the hypothesis that $G$ is $C_4$-free, one can obtain such an upper bound, as long as one has some control on the sizes of entries of $A$ (see \cref{lem:normalized case}); however, this approach cannot work if the entries of $A$ come from many very different scales. To overcome this obstruction, we apply techniques from the field of matrix scaling (see \cref{lem:scaling,lem:CD}), together with Sylvester's law of inertia (\cref{lem:sylvester}), to either reduce to the case where the entries of $A$ have roughly the same size, or to split the problem into smaller sub-problems and apply induction.

In the next section, we collect a few preliminaries we will need in the proof of \cref{thm:main}, which we prove in \cref{sec:proof main}. We end in \cref{sec:conclusion} with some concluding remarks.

\begin{remark}
    \cref{thm:inertia} shows that the inertia bound can be much weaker than the ratio bound. As suggested by Anurag Bishnoi (following an early version of this paper), it is also natural to ask the opposite question: can the ratio bound be much weaker than the inertia bound? Perhaps the easiest way to demonstrate this is using graphs of the form $G=K_{d,d}\sqcup K_{d+1}\sqcup\dots\sqcup K_{d+1}$ (i.e., the disjoint union of a $d$-regular complete bipartite graph with several copies of the $d$-regular complete graph). Indeed, if there are $d$ copies of $K_{d+1}$, then $G$ has $n=2d+d(d+1)$ vertices; one can compute that the ratio bound is $n/2$ and that the inertia bound gives the exact independence number, which is $d+(1+\dots+1)=2d=O(\sqrt n)$. More interestingly, Ihringer~\cite{Ihr23} showed that an $n$-vertex graph arising from a so-called \emph{association scheme} introduced by Cameron and Seidel~\cite{CaSe} has ratio bound $\Omega(n^{3/4})$ and inertia bound $O(n^{1/2})$ (see also \cite{deCa00} for another application of the same association scheme). This example is significant because for graphs arising from association schemes, the ratio bound coincides with two other important bounds on the independence number of a graph: the Lov\'asz theta function and Delsarte's linear programming bound (see \cite{Schr79} for details).
\end{remark}

\section{Preliminaries}

We begin with a few preliminary results, which we mostly cite from previous works. 

\subsection{Probabilistic bounds via moments}
First, we need two inequalities, due respectively to He, Zhang, and Zhang \cite{HZZ} and Zelen \cite{Zel54}, related to the general phenomenon that if the fourth moment of a random variable is commensurate to its second moment squared, then the distribution of $X$ cannot be too ``extreme''. In particular, if such a random variable has mean zero, then it must have a reasonable proportion of its probability mass on each side of zero. We remark that a simpler result with worse quantitative dependencies appears in \cite[Lemma 3.2(i)]{AGK04}, and a general approach to proving such inequalities is given in~\cite{BP05,KrNu77}.
\begin{lemma}[{\cite[Theorem 2.1]{HZZ}}]\label{lem:HZZ}
    For any real-valued random variable $Y$ and for any $y>0$, we have
    \[
    \pr(Y \geq 0) \leq 1- \frac 49 (2\sqrt 3-3) \left(-\frac{2\E[Y]}{y} + \frac{3\E[Y^2]}{y^2} - \frac{\E[Y^4]}{y^4}\right). 
    \]
\end{lemma}
One can get stronger bounds if one is given information on $\E[Y^3]$; the following such result is a special case of \cite[equation (12)]{Zel54}, specialized to $\E[Y^3]=0$.
\begin{lemma}[{\cite[equation (12)]{Zel54}}]\label{lem:zelen}
    For any real-valued random variable $Y$ with $\E[Y]=0, \E[Y^2]=1$, and $\E[Y^3]=0$, and any $a \in (-\sqrt{\E[Y^4]},\sqrt{\E[Y^4]})$, we have
    \[
    \pr(Y < a) \geq \frac{1}{2\sqrt{\E[Y^4]}\left(\sqrt{\E[Y^4]}-a\right)}.
    \]
\end{lemma}
We will apply \cref{lem:HZZ,lem:zelen} in the following form, suited to our purposes.
\begin{lemma}\label{lem:anticoncentration}
    Let $X$ be a real-valued random variable satisfying $\E[X]=0,\E[X^2]=1$, and $\E[X^4] \leq 2$. 
    \begin{lemenum}
    \item \label{lemit:anticoncentration-1} We have
    $\pr(X >0) \geq \beta$,
    where $\beta=\sqrt 3-\frac 32$ as in \cref{thm:main}.
    \item \label{lemit:anticoncentration-2} If moreover $\E[X^3]=0$, then $\pr(X>0) \geq \frac 14$.
    \end{lemenum}
\end{lemma}
\begin{proof}
    Let $Y=-X$, and note that $\E[Y]=0,\E[Y^2]=1$, and $\E[Y^4] \leq 2$. We apply \cref{lem:HZZ} with $y=2/\sqrt 3$ (which is chosen to yield the best possible bound). We find that
    \begin{align*}
        \pr(X>0) &=\pr(Y<0) =1-\pr(Y \geq 0)\geq \frac 49 (2\sqrt 3-3) \left(\frac{3}{y^2} - \frac{2}{y^4}\right)=\sqrt 3-\frac 32=\beta,
    \end{align*}
    which proves \cref{lemit:anticoncentration-1}. Under the added assumption that $\E[X^3]=0$, we have that $\E[Y^3]=0$ as well. Applying \cref{lem:zelen} with $a=0$ shows that
    \[
    \pr(X>0) = \pr(Y<0) \geq \frac{1}{2\E[Y^4]} \geq \frac 14,
    \]
    proving \cref{lemit:anticoncentration-2}.
\end{proof}

\subsection{Simple facts about eigenvalues}
We will need two basic results of linear algebra. The first is called Sylvester's law of inertia; see e.g.\ \cite[Theorem~4.5.8]{HJ13} for a proof.
\begin{lemma}\label{lem:sylvester}
    Let $A$ be a Hermitian $n \times n$ matrix, let $Z$ be an invertible $n \times n$ matrix, and let $B=ZAZ^*$. Then $A$ and $B$ have the same number of positive, negative, and zero eigenvalues. In particular, $n_{\ge 0}(A)=n_{\ge 0}(B)$.
\end{lemma}
We will also use the following simple consequence of Cauchy's interlacing formula; see e.g.\ \cite[Theorem~9.1.1]{GR01} or \cite[Theorem~4.3.28]{HJ13} for a proof. We remark that it immediately implies the inertia bound if applied to a zero principal submatrix, i.e.\ an independent set.
\begin{lemma}\label{lem:monotone}
    Let $A$ be a Hermitian matrix, and let $A'$ be a principal submatrix of $A$. Then $n_{\ge 0}(A) \geq n_{\ge 0}(A')$.
\end{lemma}

\subsection{Matrix scaling}
Finally, we need a result about matrix scaling. Matrix scaling deals with questions about when one can scale the rows and columns of a matrix in order to obtain specified row and column sums. Somewhat surprisingly, it turns out that such a scaling is possible if and only if the set of non-zero entries of the matrix satisfies certain combinatorial conditions. For more on the general matrix scaling problem, see e.g.\ the survey \cite{Idel16}.

Specifically, we use the following result, which was proved by Datta~\cite{Dat70} and subsequently reproved by Csima and Datta~\cite{CD72}. Here and for the rest of the paper, we use the notation $[n]=\{1,\dots,n\}$.
Recall that a matrix is \emph{doubly stochastic} if all its row and column sums are equal to $1$.
\begin{lemma}[{\cite{Dat70,CD72}}]\label{lem:CD}
    Let $M$ be a symmetric $n\times n$ matrix with non-negative real entries. The following two conditions are equivalent. 
    \begin{enumerate}
        \item There exists a diagonal matrix $D$ with strictly positive diagonal entries such that $DMD$ is doubly stochastic.
        \item For every $(a,b) \in [n]^2$ with $M_{ab} \neq 0$, there exists a permutation $\pi \in S_n$ with $\pi(a)=b$ such that $M_{i\pi(i)}\neq 0$ for all $i \in [n]$.
    \end{enumerate}
\end{lemma}
We will use the following simple corollary of \cref{lem:CD}. 
If $M$ is an $n \times n$ matrix and $S,T \subseteq [n]$, we denote by $M[S,T]$ the submatrix of $M$ of rows and columns in $S$ and $T$, respectively. 
\begin{lemma}\label{lem:scaling}
    Let $M$ be a symmetric $n \times n$ matrix with non-negative entries. Suppose that for all non-empty $S,T\subseteq [n]$ with $M[S,T]=0$, we have $\ab S + \ab T < n$. Then there exists a diagonal matrix $D$ with strictly positive diagonal entries such that $DMD$ is doubly stochastic.
\end{lemma}
\begin{proof}
Let $G$ be the bipartite graph with parts $A,B$, where $A=B=[n]$, and where $(x,y) \in E(G)$ if and only if $M_{xy} \neq 0$.
    Suppose for contradiction that there exists no such $D$. By \cref{lem:CD}, there exists an $(a,b) \in [n]^2$ such that $M_{ab} \neq 0$, but there is no permutation $\pi$ such that $\pi(a)=b$ and $M_{i\pi(i)} \neq 0$ for all $i$. Equivalently, the edge $(a,b)$ participates in no perfect matching of $G$.

    Now, let $H$ be the subgraph of $G$ induced on the vertex set $(A \setminus \{a\}) \cup (B \setminus \{b\})$. Then $H$ has no perfect matching. By Hall's theorem (see e.g.\ \cite[Theorem~15.15.3]{GR01}), this implies that there exists some non-empty $S \subseteq A$ with $\ab{N_H(S)}< \ab{S}$. Letting $T = (B \setminus \{b\}) \setminus N_H(S)$, we find that there is no edge between $S$ and $T$, and
    \[
    \ab S + \ab T = \ab S + ((\ab B -1) - \ab{N_H(S)}) > \ab S + ((n-1) - \ab S) = n-1.
    \]
    Hence, we have found non-empty subsets $S,T \subseteq [n]$ with $\ab S + \ab T \geq n$. Moreover, as there is no edge between $S$ and $T$ in $H$, there is also no edge between them in $G$, and hence $M_{st}=0$ for all $s \in S,t \in T$, a contradiction.
\end{proof}

\section{Proof of Theorem \ref{thm:main}}\label{sec:proof main}
In this section, we prove \cref{thm:main}. We begin by proving the following lemma, which implies \cref{thm:main} in case all rows of the weighted adjacency matrix have the same $L^2$ norm. Later, we will see how this special case actually implies the full theorem.
\begin{lemma}\label{lem:normalized case}
    Let $G$ be an $n$-vertex graph, and let $B$ be a Hermitian weighted adjacency matrix of $G$. Suppose that every row of $B$ has $L^2$ norm $1$, i.e.\ that $\sum_{j=1}^n \ab{B_{ij}}^2=1$ for all $i \in [n]$.
    \begin{lemenum}
        \item\label{lemit:C4-free} If $G$ is $C_4$-free, then $n_{\geq 0}(B) \geq \beta n$, where $\beta=\sqrt 3-\frac 32$ as in \cref{thm:main}.
        \item\label{lemit:girth-5} If $G$ has girth at least $5$, then $n_{\geq 0}(B) \geq n/4$.
    \end{lemenum}
\end{lemma}
\begin{proof}
    We begin with \cref{lemit:C4-free}, so let $G$ be a $C_4$-free graph. Let $\lambda_1 \geq \dotsb \geq \lambda_n$ be the eigenvalues of $B$. Let $X$ be the random variable that takes value $\lambda_i$ with probability $1/n$, for all $i \in [n]$. We have the identity 
    \begin{equation}\label{eq:identity}
        n_{\ge 0}(B) = n \cdot \pr(X \geq 0) \geq n \cdot \pr(X>0).
    \end{equation}
    Since $B$ has zero diagonal, we know that $\E[X] = \tr(B)/n = 0$. Similarly, by the assumption that the rows of $B$ have $L^2$ norm $1$, we can compute
    \[
    \E[X^2] = \frac 1n \tr(B^2) =\frac 1n\sum_{i=1}^n \sum_{j=1}^n B_{ij} B_{ji}= \frac 1n \sum_{i=1}^n\sum_{j=1}^n \ab{B_{ij}}^2 = 1.
    \]
    We would like to apply \cref{lem:anticoncentration}, so we need to estimate $\E[X^4]$, which is equal to $\frac 1n \tr(B^4)$. Note that \[\tr(B^4)=\sum_{i,j,k,\ell=1}^nB_{ij}B_{jk}B_{k\ell}B_{\ell i}.\]Since $B$ is a weighted adjacency matrix of $G$, a tuple of vertices $(i,j,k,\ell)$ can provide a nonzero contribution only if it represents a closed walk in $G$.

    We now use the fact that $G$ is $C_4$-free to find that the only closed walks of length $4$ which provide a nonzero contribution to $\tr(B^4)$ are \emph{degenerate} (i.e.\ they repeat vertices). In $G$, such walks come in two combinatorial types: those that traverse a single edge four times, and those that traverse a \emph{cherry} (i.e.\ a 2-edge path) twice. Each single-edge walk is counted twice, once for each of its starting vertices. Each cherry walk is counted four times: once starting at each of its outer vertices, and twice starting at its central vertex, since from there the cherry can be traversed in two different orientations.

    We can sum over (unlabeled) cherries by first summing over choices for the central vertex, then summing over ordered pairs of candidates for the outer vertices, then subtracting the contribution from non-distinct candidates for the outer vertices (and then dividing by 2 to account for the fact that the outer vertices should not be ordered). This yields
    \begin{align*}
        \tr(B^4)&=\sum_{i=1}^n \sum_{j=1}^n \ab{B_{ij}}^4+4\cdot \frac12\left(\sum_{i=1}^n \left(\sum_{j=1}^n \ab{B_{ij}}^2\right)^2 - \sum_{i=1}^n \sum_{j=1}^n \ab{B_{ij}}^4\right)\\
        &= 2\sum_{i=1}^n \left(\sum_{j=1}^n \ab{B_{ij}}^2\right)^2 - \sum_{i=1}^n \sum_{j=1}^n \ab{B_{ij}}^4\\
        &= 2n - \sum_{i=1}^n \sum_{j=1}^n \ab{B_{ij}}^4 \leq 2n.
    \end{align*}
    Therefore,
    \[
    \E[X^4] = \frac 1n \tr(B^4) \leq 2.
    \]
    We may therefore apply \cref{lemit:anticoncentration-1} to find that $\pr(X>0)\geq \beta$.
    Plugging this into \eqref{eq:identity} completes the proof of \cref{lemit:C4-free}. Then, \cref{lemit:girth-5} can be proved in virtually the same way, using \cref{lemit:anticoncentration-2} instead of \cref{lemit:anticoncentration-1}. Indeed, note that if $G$ has girth at least 5 then it has no closed walks of length 3, so $\E[X^3]=0$.
\end{proof}
We are now ready to prove \cref{thm:main}.
\begin{proof}[Proof of \cref{thm:main}]
We begin with \cref{thmit:C4-free}.
    Our proof is by induction on $n$, with the base case $n=1$ being trivial as then $A$ must be the zero matrix. We now assume that $n>1$, and that the result has been proved for all smaller values of $n$.

    We split into two cases according to \cref{lem:scaling}. Suppose first that there exist non-empty sets $S,T \subseteq [n]$ with $\ab S + \ab T \geq n$ and $A[S,T]=0$. As $A$ is Hermitian, this implies that $A[T,S]=0$ as well. Let us permute the rows of $A$ so that the rows in $S \setminus T$ come first, then those in $S \cap T$, then those in $T \setminus S$, and finally those in $[n] \setminus (S \cup T)$ (and apply the same permutation to the columns, so that $A$ remains Hermitian). Then we can write $A$ in block form as
    \[
    \begin{blockarray}{ccccc}
        &S \setminus T & S \cap T & T \setminus S&\\
        \begin{block}{c(cccc)}
            S \setminus T&A_1 &0&0&*\\
            S \cap T & 0 & 0 & 0 & *\\
            T \setminus S & 0 &0 & A_2 &*\\
            &* &*&*&*\\
        \end{block}
    \end{blockarray}
    ~,
    \]
    where $A_1,A_2$ are principal submatrices of $A$ of sizes $\ab{S \setminus T},\ab{T \setminus S}$, respectively, and where $*$ denotes an arbitrary submatrix. By \cref{lem:monotone}, we know that $n_{\ge 0}(A) \geq n_{\ge 0}(A')$, where $A'$ is the principal submatrix on rows and columns in $S \cup T$. Moreover, as $A'$ is block-diagonal, its eigenvalues are those of $A_1$, those of $A_2$, plus $\ab{S \cap T}$ additional zero eigenvalues. Applying the inductive hypothesis to $A_1$ and $A_2$, which are Hermitian weighted adjacency matrices of smaller $C_4$-free graphs, we find that
    \begin{align*}
        n_{\ge 0}(A) &\geq n_{\ge 0}(A') \\
        &= n_{\ge 0}(A_1) + n_{\ge 0}(A_2) + \ab{S \cap T}\\
        &\geq \beta{\ab{S \setminus T}} + \beta{\ab{T \setminus S}} + \ab{S \cap T}\\
        &\geq \beta \left({\ab{S \setminus T} + \ab{T \setminus S} + 2\ab{S \cap T}}\right) &&[\text{as }\beta \leq \tfrac 12]\\
        &= \beta(\ab S + \ab T)\\
        &\geq \beta n.
    \end{align*}
    This concludes the proof in this case. Therefore, we may assume that such sets $S,T$ do not exist. Let $M$ be the matrix defined by $M_{ij} = \ab{A_{ij}}^2$, and note that $M$ has non-negative entries and the same set of zero entries as $A$. In particular, if $M[S,T]=0$ for some $S,T \subseteq [n]$, then $\ab S + \ab T<n$. By \cref{lem:scaling}, we conclude that there exists a diagonal matrix $D$ with strictly positive diagonal entries such that $DMD$ is doubly stochastic. Let the diagonal entries of $D$ be $d_1,\dots,d_n$.

    Now, let $Z$ be the diagonal matrix whose diagonal entries are $\sqrt{d_1},\dots,\sqrt{d_n}$, and let $B=Z A Z=ZAZ^*$. By \cref{lem:sylvester}, we have that $n_{\ge 0}(B) = n_{\ge 0}(A)$. Moreover, the $(i,j)$ entry of $B$ is $\sqrt{d_i d_j} A_{ij}$, and hence for any $i \in [n]$,
    \begin{equation*}
    \sum_{j=1}^n \ab{B_{ij}}^2 = \sum_{j=1}^n d_i d_j \ab{A_{ij}}^2 = \sum_{j=1}^n d_i d_j M_{ij} = \sum_{j=1}^n (DMD)_{ij} = 1,
    \end{equation*}
    since the matrix $DMD$ is doubly stochastic. We may now apply \cref{lemit:C4-free} to $B$ to conclude that $n_{\geq 0}(A)=n_{\geq 0}(B) \geq \beta n$, which completes the proof. The proof of \cref{thmit:girth-5} is again virtually identical; one simply replaces the application of \cref{lemit:C4-free} above by \cref{lemit:girth-5}.
\end{proof}

\section{Concluding remarks}\label{sec:conclusion}
To end this paper, we would like to highlight two further questions about the inertia bound that our techniques seem incapable of resolving.
The first is a question of Ihringer \cite{Ihr19}, who asked whether the inertia bound is tight for almost all graphs. This is equivalent to asking about the tightness of the inertia bound for the binomial random graph $\mathbb G(n,\frac 12)$. We make the following conjecture.
\begin{conjecture}\label{conj:random}
    The following holds with probability $1-o(1)$ for $G \sim \mathbb G(n, \frac 12)$ as $n \to \infty$. Every weighted adjacency matrix $A$ of $G$ has $n_{\geq 0}(A) =\Omega(\frac{n}{\log n})$.
\end{conjecture}
In particular, as $\alpha(\mathbb G(n,\frac 12)) = O(\log n)$ with probability $1-o(1)$, \cref{conj:random} would imply that the inertia bound is far from tight for almost all graphs. An earlier version of this paper conjectured a stronger bound, namely that $n_{\geq 0}(A) \geq (\frac 12 - o(1))n$ for every weighted adjacency matrix of $G \sim \mathbb G(n, \frac 12)$. However, Noga Alon pointed out to us that such a statement is false, and that \cref{conj:random} would be best possible if true.  Alon observed that for every graph $G$, there exists a weighted adjacency matrix $A$ with $n_{\geq 0}(A)=\chi(\overline G)$, where $\chi(\overline G)$ is the clique cover number of $G$. Indeed, one can easily check that the adjacency matrix of a disjoint union of $m$ cliques has exactly $m$ non-negative eigenvalues, which implies the claim by taking $A$ to be the adjacency matrix of a subgraph of $G$ which forms an optimal clique cover. A famous result of Bollob\'as \cite{MR0951992} implies that $\chi(\overline G) = O(\frac{n}{\log n})$ with probability $1-o(1)$ for $G \sim \mathbb G(n, \frac 12)$, showing that \cref{conj:random} would be best possible if true.

\cref{thm:main} (together with results on the $C_4$-free process \cite{BoKe}) implies that there exist graphs $G$ for which $n_{\geq 0}(A)$ is larger than $\alpha(G)^{\frac 32-o(1)}$ for any weighted adjacency matrix $A$ of $G$. If \cref{conj:random} is true, it would imply that there exist graphs with $n_{\geq 0}(A) > 2^{\alpha(G)/3}$ for any weighted adjacency matrix $A$. However, neither of these rule out the possibility that $n_{\geq 0}(A)$ is upper-bounded by \emph{some} function of $\alpha(G)$. We conjecture that no such upper bound exists.
\begin{conjecture}\label{conj:alpha 2}
    For every integer $k$, there exists a graph $G$ with $\alpha(G)=2$ and $n_{\geq 0}(A) \geq k$ for every weighted adjacency matrix $A$ of $G$.
\end{conjecture}
We remark that a result of Konyagin \cite{Kon81} implies the analogous statement for the ratio bound, i.e.\ that for every $k\in \mathbb N$ there exists a graph $G$ with $\alpha(G)=2$ for which the ratio bound (with any weighted adjacency matrix) cannot be used to prove $\alpha(G)<k$. Indeed, Konaygin showed that there are graphs $G$ with $\alpha(G)=2$ which have arbitrarily large values of the so-called \emph{Lov\'asz theta number} \cite{Lov79}. For more details, see \cite{Alon19} or \cite[Chapter 11]{Lov19}.

\paragraph{Acknowledgements:} We are grateful to Noga Alon, Anurag Bishnoi, Clive Elphick, and Ferdinand Ihringer for helpful comments and interesting discussions on earlier drafts of this paper.

\providecommand{\bysame}{\leavevmode\hbox to3em{\hrulefill}\thinspace}
\providecommand{\MR}{\relax\ifhmode\unskip\space\fi MR }
\providecommand{\MRhref}[2]{%
  \href{http://www.ams.org/mathscinet-getitem?mr=#1}{#2}
}
\providecommand{\href}[2]{#2}

\end{document}